\documentclass[12pt,reqno]{amsart}
\usepackage{geometry,amsmath,amssymb,amsthm,amsfonts,enumerate,hyperref,comment}
\usepackage{verbatim}
\usepackage{tensor}

\newgeometry{vmargin={40mm},hmargin={30mm,30mm}}

\theoremstyle{plain}

\newtheorem*{theorem*}{Theorem}
\newtheorem{theorem2}{Theorem}[section]
\newtheorem{proposition}[theorem2]{Proposition}

\newtheorem{lemma}[theorem2]{Lemma}

\theoremstyle{definition}
\newtheorem{definition}[theorem2]{Definition}

\theoremstyle{remark}
\newtheorem*{remark}{Remark}

\numberwithin{equation}{section}

\newcommand{\tr}{\mathrm{tr}}

\renewcommand{\div}{\mathrm{div}}

\begin{document}
\title{On the Rigidity of $\mathbb{CP}^{2n}\times \mathbb{CP}^{1}$}
\author{Stuart James Hall}
\address{School of Mathematics and Statistics, Herschel Building, Newcastle University, Newcastle-upon-Tyne, NE1 7RU} 
\email{stuart.hall@ncl.ac.uk}
\urladdr{\href{https://www.ncl.ac.uk/maths-physics/staff/profile/stuarthall}{https://www.ncl.ac.uk/maths-physics/staff/profile/stuarthall}}

\begin{abstract}
	We revisit Koiso's original examples of rigid infinitesimally deformable Einstein metrics. We show how to compute Koiso's obstruction to the integrability of the infinitesimal deformations on $\mathbb{CP}^{n}\times \mathbb{CP}^{1}$ using completely elementary complex differential geometry.  
\end{abstract}
\maketitle
\section{Introduction}\label{Sec:1}
In the early 1980s, Koiso developed the structure theory for the moduli space of Einstein metrics around a given Einstein metric $g$ on a closed manifold $\mathcal{M}$ (see \cite{KoiOsaka1},  \cite{KoiOsaka2}, and \cite{Koi_EMCS}). The given metric $g$ is said to be \textit{infinitesimally deformable} if one can find tensors that look like the tangents to putative one-parameter families of geometrically distinct Einstein metrics passing through $g$ (see the next section for more detail). In \cite{KoiOsaka2}, Koiso found an obstruction to being able to integrate an infinitesimal deformation to form a genuine one-parameter family of Einstein metrics. He considered the case of symmetric spaces and proved the following result (which we paraphrase).
\begin{theorem*}[Koiso \cite{KoiOsaka2}, Theorem 6.12]
	The product K\"ahler--Einstein metric on ${\mathbb{CP}^{2n}\times \mathbb{CP}^{1}}$ admits infinitesimal Einstein deformations, none of which is integrable to second order.
\end{theorem*}
This gave the first exampless of Einstein metrics that admitted infinitesimal deformations but were nevertheless isolated in the moduli space; metrics that are isolated are referred to as being \textit{rigid}. \\
\\
In \cite{KoiOsaka1}, Koiso demonstrated that the K\"ahler--Einstein metric on a complex Grassmannian of $m$-dimensional subspaces of an $(n+m)$-dimensional vector space admits infinitesimal Einstein deformations when $1<m,n$ (i.e. when the Grassmannian is not $\mathbb{CP}^{n}$). Recent work of the author with Schwahn and Semmelmann \cite{HSS24} has proved that, in the case when $n+m$ is odd, none of the deformations is integrable to second order and so the metric is rigid\footnote{Prior to producing the joint work \cite{HSS24}, the results appeared in the preprints \cite{Hall24} and \cite{SchSem24}.} (the $m=2$ case of this result was proved by Nagy and Semmelmann \cite{NagSem23}). The product metric on $\mathbb{CP}^{n}\times \mathbb{CP}^{1}$ is K\"ahler--Einstein. In this article we give another proof of Koiso's Theorem using the `local' complex-geometric approach that was applied to the Grassmannians in \cite{HSS24}. 

\subsection{Differences from Koiso's original approach}
The infinitesimal deformations of the product $\mathbb{CP}^{n}\times \mathbb{CP}^{1}$ were described explicitly by Koiso (note we do not require the complex dimension of the first factor to be even here). We use the fact that product metric is K\"ahler to verify the construction yields infinitesimal deformations. Koiso computed his obstruction for these tensors and showed that it did not vanish, thus showing all the deformations are obstructed. The calculation of the obstruction involves various integral quantities that are `cubic' in the deformation tensors (see Section \ref{Sec:2}). Koiso directly manipulated the expressions that arise in an impressive feat of computation, turning on the bread-and-butter techniques one learns as a student of Riemannian geometry.\\
\\
We use a different method, which is actually implicit in Koiso's paper \cite{KoiOsaka2}, and has been used recently to resolve similar rigidity questions. We give a brief outline of the principle here but refer the reader to \cite{BHMW} and \cite{HSS24} for further details. The group $G=SU_{n+1}$ acts isometrically on $\mathbb{CP}^{n}\times \mathbb{CP}^{1}$ by its standard action on the first factor of the product. The space of infinitesimal Einstein deformations is equivariantly isometric to $\mathfrak{g}=\mathfrak{su}_{n+1}$ and the terms in the obstruction can be thought of as $G$-invariant homogeneous cubic polynomials in $\mathfrak{g}$. The key to the calculations is that the space of such polynomials is $1$-dimensional and thus all the terms can be realised as a multiple of a fixed generator. To find the multiplying factor, we pick an element of $\mathfrak{g}$ that is particularly easy to work with, compute both the obstruction term and the generating polynomial, and compare. It is possible to make the calculations using some elementary complex geometry. In the case that the complex dimension of the first factor is even, Koiso's result follows from the fact that the generating polynomial does not vanish.  \\
\\
\textit{Acknowledgements:}  The author would like to thank Paul Schwahn and Uwe Semmelmann for useful discussions.
\section{Einstein deformations}\label{Sec:2}

\subsection{General Theory}

A good general exposition of the theory is given in \cite{Bes}. For brevity, we simply state the equations that are satisfied by an infinitesimal Einstein deformation ${h\in \Gamma(s^{2}(T^{\ast}\mathcal{M}))}$ 

\begin{definition}[Infinitesimal Einstein Deformation]
	Let $(\mathcal{M}^{n},g)$ be an Einstein manifold such that ${\mathrm{Ric}(g)=\lambda g}$ with $\lambda>0$. An \textit{infinitesimal Einstein deformation} (EID) is a section $h\in \Gamma(s^{2}(T^{\ast}\mathcal{M}))$ satisfying the following conditions:
	\begin{eqnarray}
	\mathrm{tr}_{g}(h) = 0, \label{eqn:trace_free}\\
	\div(h)=0, \label{div_free}\\
	 \Delta h +2\mathrm{Rm}(h) = 0, \label{Lin_Ein}
	\end{eqnarray}
	where $\Delta$ is the connection Laplacian and $\mathrm{Rm}$ is the curvature operator acting on symmetric 2-tensors. We denote the space of EIDs by $\varepsilon(g)$. 
\end{definition}

If the manifold $(\mathcal{M}^{n},g)$ is K\"ahler--Einstein and the EID $h$ is also invariant with respect to the complex structure $J$, then the equations definining an EID can be conveniently expressed as conditions involving the two-form associated to $h$, $\sigma \in \Omega^{(1,1)}(\mathcal{M})$:
\begin{eqnarray}
\Lambda(\sigma)=0, \label{eqn:HD1} \\
\bar{\partial}^{\ast}\sigma=0, \label{eqn:HD2} \\
\Delta_{\bar{\partial}}\sigma = \lambda \sigma, \label{eqn:HD3}     
\end{eqnarray}
where $\Lambda$ is the adjoint of the Lefschetz operator, $\bar{\partial}^{\ast}$ is the usual adjoint of the Dolbeault operator $\bar{\partial}$, and ${\Delta_{\bar{\partial}}=\bar{\partial}\bar{\partial}^{\ast}+\bar{\partial}^{\ast}\bar{\partial}}$.\\ 
\\
If $h\in \varepsilon(g)$, the path $g(t)=g+th$ solves the Einstein equation to first order at $t=0$ (i.e. $h$ solves the linearised Einstein equation). For a given EID, Koiso developed an obstruction to being able to produce a curve of Einstein metrics solving the Einstein equations to second order; again, for brevity, we only state the obstruction.

\begin{lemma}[Koiso, Lemma 4.3 in \cite{KoiOsaka2}]
   	Let $(\mathcal{M},g)$ be an Einstein metric with Einstein constant $\lambda>0$ and let $h \in \varepsilon(g)$. Then an obstruction to the integrability of $h$ to order two is given by the nonvanishing of the quantity
   	\begin{small}
   		\begin{equation}\label{eqn:Koiobs}
   		\mathcal{I}(h) :=2\lambda\langle h_{i}^{k}h_{kj},h_{ij}\rangle_{L^{2}}+3\langle\nabla_{i}\nabla_{j}h_{kl},h_{ij}h_{kl} \rangle_{L^{2}}-6\langle\nabla_{i}\nabla_{l}h_{kj},h_{ij}h_{kl}\rangle_{L^{2}},  
   		\end{equation}
   	\end{small}
   	where each of the brackets denotes the $L^{2}$-inner product induced by the metric $g$ on the appropriate bundle.
   \end{lemma}
\begin{remark}
Nagy and Semmelmann \cite{NagSem23} have reformulated the obstruction (\ref{eqn:Koiobs}) in a coordinate free manner using the Fr\"olicher--Nijenhuis bracket; one could also use this approach to compute $\mathcal{I}(h)$ and reprove Koiso's Theorem.
\end{remark}
\subsection{The infinitesimal deformations for $\mathbb{CP}^{n}\times \mathbb{CP}^{1}$}

Koiso completely characterised the infinitesimal deformations of symmetric spaces. In particular, he proved that for the product metric $g_{1}\oplus g_{0}$ on  $\mathbb{CP}^{n}\times \mathbb{CP}^{1}$,
\[
\varepsilon(g_{1}\oplus g_{0})\cong \mathfrak{su}_{n+1}.
\]
Of course $\mathfrak{su}_{n+1} \cong \mathfrak{iso}(g_{1})$ and, as $g_{1}$ is a K\"ahler-Einstein metric, we also have by Matsushima's Theorem $\mathfrak{iso}(g_{1})\cong E_{1}$, where $E_{1}$ is the eigenspace of the smallest non-zero eigenvalue of the ordinary Laplacian; here the smallest eigenvalue will be $2\lambda$ where $\lambda$ is the Einstein constant. We give an explicit description of the EIDs in terms of the eigenfunctions and include a brief proof that the tensors satisfy the Equations (\ref{eqn:HD1}), (\ref{eqn:HD2}), and (\ref{eqn:HD3}). However, we will omit the proof that these are \textit{all} the EIDs.

\begin{proposition}[cf. Koiso \cite{KoiOsaka2}, Theorem 5.7]
Let $g=g_{1}\oplus g_{0}$ be the product K\"ahler--Einstein metric on $\mathbb{CP}^{n}\times \mathbb{CP}^{1}$ with Einstein constant $\lambda>0$ and let $f$ be an eigenfunction of the ordinary Laplacian on $(\mathbb{CP}^{n},g_{1})$ with eigenvalue $2\lambda$. Then 
\[
h = \mathrm{Hess}(f)+\lambda f{g}_{1} + (1-n)\lambda f g_{0}
\] 	
is an EID.	
\end{proposition}
\begin{proof}
To check Equation (\ref{eqn:trace_free})	(or (\ref{eqn:HD1})) we take the trace of $h$
\[
\mathrm{tr}_{g}(h) = -2\lambda f+2n\lambda f+2(1-n)\lambda f = 0.
\]	
For (\ref{eqn:HD2}), let $\sigma$ be the $(1,1)$-form associated to $h$; if the K\"ahler forms of $g_{1}$ and $g_{0}$ are $\omega_{1}$ and $\omega_{0} $, then  
\[
\sigma = i\partial\bar{\partial}{f}+\lambda f\omega_{1}+ (1-n)f\omega_{0}.
\] 
Taking the codifferential yields
\[
\bar{\partial}^{\ast}\sigma = i\bar{\partial}^{\ast}\partial\bar{\partial}{f} +\lambda\bar{\partial}^{\ast} (f\omega_{1})+ (1-n)\lambda \bar{\partial}^{\ast}(f\omega_{0}).
\]
If we denote the differential and codifferential operators associated to the metrics $g_{j}$ by a subscript $j$, we have 
\[
\bar{\partial}^{\ast}\sigma =  i\bar{\partial}^{\ast}_{1}\partial_{1}\bar{\partial}_{1}{f} +\lambda\bar{\partial}^{\ast}_{1}(f\omega_{1})+ (1-n)\lambda f \bar{\partial}^{\ast}_{0}(\omega_{0}).
\]
Let $L(-) = -\wedge \omega$ denote the usual Lefschetz operator. Using the K\"ahler identities $\bar{\partial}^{\ast}\partial=-\partial\bar{\partial}^{\ast}$ and $[\bar{\partial}^{\ast},L]=i\partial$, we compute
\[
\bar{\partial}^{\ast}\sigma = -i\partial_{1} \bar{\partial}^{\ast}_{1}\bar{\partial}_{1}f + i\lambda\partial_{1} f + 0 = -i\lambda \partial_{1} f +i\lambda \partial_{1} f+0=0, 
\]
where we have used $ \bar{\partial}^{\ast}_{1}\bar{\partial}_{1}f = \frac{1}{2}\Delta_{1}f=\lambda f$.\\
\\
For Equation (\ref{eqn:HD3}) we can compute
\[
\Delta_{\bar{\partial}}\sigma = \Delta_{\bar{\partial}_{1}}(i\partial_{1}\bar{\partial}_{1}f)+\lambda \Delta_{\bar{\partial}_{1}}(L_{1}(f))+(1-n)\lambda\Delta_{\bar{\partial}}(f\omega_{0})
\]
Applying the K\"ahler identities (the Laplacian commutes with every operator) along with the fact that, for any smooth function on the first factor $F:\mathbb{CP}^{n}\rightarrow \mathbb{R}$,
\[
\Delta_{\bar{\partial}}(F\omega_{0}) = (\Delta_{\bar{\partial}_{1}}F)\omega_{0},
\]
yields the result.
\end{proof}

\section{Objects in local coordinates}
We consider $\mathbb{CP}^{n}$ with the standard dense open set $U_{0}$, described in homogeneous coordinates by 
\[
U_{0}=\left\{[z_{0}:z_{1}:\ldots:z_{n}] \in \mathbb{CP}^{n} :z_{0} \neq 0 \right\}.
\]
 Let ${\pi:U_{0}\rightarrow \mathbb{C}^{n}}$ be the coordinate chart given by
\[
\pi([z_{0}:z_{1}:\ldots:z_{n}]) = \left( \frac{z_{1}}{z_{0}}, \frac{z_{2}}{z_{0}}, \ldots, \frac{z_{n}}{z_{0}} \right),
\]
and denote the corresponding holomorphic coordinates ${(w_{1},w_{2},\ldots, w_{n})\in \mathbb{C}^{n}}$, where ${w_{i}=z_{i}/z_{0}}$. We will use the notation
\[
\|w\|^{2} = w_{1}^{2}+w_{2}^{2}+\cdots +w_{n}^{2}.
\]
We can write the metric and some associated geometric objects in the $w$ coordinates. This material is completely standard and so we omit the proof.
\begin{lemma}[Metric objects in coordinates]
In the $w$ coordinates, the Fubini--Study metric $g$ is given by
\begin{equation}\label{eqn:FS_met}
g_{k\bar{l}} = \frac{1}{1+\|w\|^{2}}\left(\delta_{kl}-\frac{\overline{w}_{k}w_{l}}{1+\|w\|^{2}} \right);
\end{equation} 
the inverse metric $g^{-1}$ (in the sense that ${g^{k\bar{l}}g_{j\bar{l}} = \delta_{kj}}$) is given by
\begin{equation}\label{eqn:FS_inv}
g^{k\bar{l}} = \left(1+\|w\|^{2}\right)\left( \delta_{kl} +w_{k}\overline{w}_{l}\right);
\end{equation}
the volume form $dV_{g}$ is given by
\begin{equation}\label{eqn:volform}
dV_{g} = \frac{1}{(1+\|w\|^{2})^{n+1}}\left(\frac{\sqrt{-1}}{2}\right)^{n}dw_{1}\wedge d\bar{w}_{1}\wedge dw_{2}\wedge d\bar{w}_{2}\wedge\cdots\wedge dw_{n}\wedge d\bar{w}_{n};
\end{equation}
the Christoffel symbols are given by
\begin{equation}\label{eqn:Christ}
\Gamma_{ij}^{k} = -\frac{1}{1+\|w\|^{2}}\left(\delta_{jk}\overline{w}_{i}+\delta_{ik}\overline{w}_{j} \right);
\end{equation}
the curvature tensor is given by
\begin{equation}\label{eqn:curvten}
\mathrm{Rm}^{\;j\;l}_{i\;k} =  (\delta_{i}^{j}\delta_{k}^{l}+\delta_{i}^{l}\delta_{k}^{j});
\end{equation}
the Ricci tensor is given by
\begin{equation}
\mathrm{Ric}(g)_{k\bar{l}} = (n+1)g_{k\bar{l}}.	
\end{equation}	
\end{lemma}
If we denote $E_{1}$ the space of eigenfunctions for the ordinary Laplacian with eigenvalue $2(n+1)$, there is an isomorphism ${F:\mathfrak{su}(n+1)\rightarrow E_{1}}$ given in coordinates by
\[
F(\eta)(w) = \frac{\overline{W}\eta W^{t}}{1+\|w\|^{2}},   
\]
where $W=(1,w_{1},w_{2},\ldots,w_{n})$. Where revelvant, for $\eta \in \mathfrak{g}$, we will denote the eigenfunction $F(\eta)$ by $f_{\eta}$. Once and for all, we fix $\gamma \in \mathfrak{su}(n+1)$ to be
\begin{equation}\label{Eqn:Gamma_choice}
\gamma=\mathrm{Diag}(-n,\underbrace{1,1,\ldots,1}_{n \ \mathrm{terms}}).
\end{equation}
 We collect the coordinate expressions for objects associated with $f_{\gamma}$, leaving proofs to the reader.
\begin{lemma}[Objects associated with $f_{\gamma}$ in coordinates]
In the $w$ coordinates, the eigenfunction $f_{\gamma}$ is given by
\begin{equation}\label{eqn:eigenfunc}
f_{\gamma}(w) = \frac{\|w\|^{2}-n}{1+\|w\|^{2}};
\end{equation}
the derivative $\partial f_{\gamma}$ is given by
\begin{equation}
\left(\partial f_{\gamma}\right)_{k} = (n+1)\frac{\overline{w}_{k}}{(1+\|w\|^{2})^{2}};
\end{equation}
the Hessian of $f_{\gamma}$ is given in complex coordinates by
\begin{equation} \label{eqn:Hess_com_1} 
\nabla_{k}\nabla_{\bar{l}}f_{\gamma} = \frac{n+1}{(1+\|w\|^{2})^{2}}\left(\delta_{kl}-\frac{2\overline{w}_{k}w_{l}}{1+\|w\|^{2}} \right) = \frac{n+1}{1+\|w\|^{2}}\left(2g_{k\bar{l}}-\frac{\delta_{kl}}{1+\|w\|^{2}}\right), 
\end{equation}
\begin{equation}\label{eqn:Hess_com_2} 
 \nabla_{k}\nabla_{l}f_{\gamma} = 0;    
\end{equation}
the covariant derivative of the Hessian is given by
\begin{equation}\label{eqn:3Der}
\nabla_{\bar{q}}\nabla_{k}\nabla_{\bar{l}}f_{\gamma} = -\left(g_{k\bar{l}}(\partial f_{\gamma})_{\bar{q}}+g_{k\bar{q}}(\partial f_{\gamma})_{\bar{l}}\right);
\end{equation}
the four derivative term needed in Koiso's obstruction is given by
\begin{equation}\label{eqn:4Der}
\nabla_{{p}}\nabla_{\bar{q}}\nabla_{k}\nabla_{\bar{l}}f_{\gamma} =-(g_{k\bar{l}}\nabla_{p}\nabla_{\bar{q}}f_{\gamma}+g_{k\bar{q}}\nabla_{p}\nabla_{\bar{l}}f_{\gamma}).
\end{equation}	
\end{lemma}

Finally in this section, we note that we can write the terms in Koiso's obstruction (\ref{eqn:Koiobs}) in complex coordinates but this will introduce some conversion factors in the terms. We refer the reader to \cite{Hall24} for a proof of this Lemma.  The tangent space at a point is an inner product space $(V^{2n},g)$ with an almost complex structure $J$ such that $g(J\cdot,J\cdot)=g(\cdot,\cdot)$. We complexify $V$ and extend the tensors $g$ and $h$ $\mathbb{C}$-linearly in both arguments to obtain a $2n$-complex-dimensional space $V_{\mathbb{C}}$ and tensors $g_{\mathbb{C}}$ and $h_{\mathbb{C}}$. The space $V_{\mathbb{C}}$ splits into the $\pm \sqrt{-1}$-eigenspaces for $J$ and we write ${V_{\mathbb{C}}=V_{\mathbb{C}}^{(1,0)}\oplus V_{\mathbb{C}}^{(0,1)}}$.\\
\\
Given a $g$-orthonormal basis of $V$ of the form $v_{1},Jv_{1},v_{2},Jv_{2},\ldots, v_{n},Jv_{n}$, we can form the basis $\{e_{i}\}_{i=1}^{n}$ of $V^{(1,0)}_{\mathbb{C}}$ where   
\[
e_{i}=\frac{1}{2}\left(v_{i}-\sqrt{-1}Jv_{i} \right).
\]
This is  an orthogonal basis of $(V^{(1,0)}_{\mathbb{C}},g_{\mathbb{C}})$ with $\|e_{i}\|=1/2$. The set of conjugates 
\[
\bar{e}_{i}=\frac{1}{2}\left(v_{i}+\sqrt{-1}Jv_{i} \right),
\]
form a basis of $V^{(0,1)}_{\mathbb{C}}$. As the tensors $g$ and $h$ are $J$-invariant, the only non-vanishing terms of the extensions  $g_{\mathbb{C}}$ and $h_{\mathbb{C}}$ are those of the form 
\[
g_{k\bar{l}}:=g_{\mathbb{C}}(e_{k},\bar{e}_{l}) \qquad \mathrm{and} \qquad h_{k\bar{l}} = h_{\mathbb{C}}(e_{k},\bar{e}_{l}). 
\] 
We consider Koiso's quantities but in complex coordinates e.g.
\[
\langle h_{k}^{p}h_{p\bar{l}},h_{k\bar{l}} \rangle  = g^{k\bar{q}}g^{r\bar{l}}h_{k}^{p}h_{p\bar{l}}h_{r\bar{q}} = H_{k}^{p}H_{p}^{r}H_{r}^{k} = \tr(H^{3}).
\]
As the metric is K\"ahler, the Chern connection is the same as the $\mathbb{C}$-linear extension of the Levi-Civita connection. We also note that 
\[
(\nabla_{\cdot}\nabla_{\cdot}h)(JX,JY) = (\nabla_{\cdot}\nabla_{\cdot}h)(X,Y).
\]

\begin{lemma}\label{Lem:ComptoRiem}
	Let $h\in s^{2}(V^{\ast})$ be $J$-invariant and let $T\in (V^{\ast})^{\otimes 4}$ satisfy
	\[
	T(-,-,X,Y) = T(-,-,Y,X) \qquad \mathrm{and} \qquad T(-,-,JX,JY)=T(-,-,X,Y),
	\]
	for all $X,Y \in V$. Then, with the notation defined previously,
	\begin{eqnarray}
	\langle h_{kp}h^{p}_{l},h_{kl} \rangle = 2\langle h_{k\bar{p}}h^{\bar{p}}_{\bar{l}},h_{k\bar{l}} \rangle, \label{RiemCom_1}\\
	\langle T_{klrs},h_{kl}h_{rs} \rangle = 4 \mathrm{Re}(\langle T_{k\bar{l}r\bar{s}},h_{k\bar{l}}h_{r\bar{s}} \rangle),
	\label{RiemCom_2}\\
	\langle T_{krsl}, h_{kl}h_{rs} \rangle = 2 \mathrm{Re}(\langle T_{k\bar{r}s\bar{l}} ,h_{k\bar{l}}h_{s\bar{r}}\rangle). \label{RiemCom_3}
	\end{eqnarray}
	
\end{lemma}

\section{Computing integral terms}
Throughout this section we will consider the Fubini--Study metric $g$ on $\mathbb{CP}^{n}$ with the scaling of the previous section such that the Einstein constant is $(n+1)$. All of the integrals we need to calculate are of a similar form; more specifically, for $r\in \{0,1,2,3\}$, we will need formulae for the integrals
\begin{equation}\label{eqn:Integral}
I_{n}^{r} :=\int_{\mathbb{C}^{n}}\frac{\|w\|^{2r}}{(1+\|w\|^{2})^{n+4}}\left(\frac{\sqrt{-1}}{2}\right)^{n}dw_{1}\wedge d\bar{w}_{1}\wedge \cdots \wedge dw_{n}\wedge d\bar{w}_{n}.
\end{equation}

\begin{lemma}[Standard integral values] \label{lem:SIV}
Let $I_{n}^{r}$ be as in (\ref{eqn:Integral}), then
	\[
	I_{n}^{0} = 6\cdot\frac{\pi^{n}}{(n+3)!}, \qquad I_{n}^{1} =2n\cdot\frac{\pi^{n}}{(n+3)!}, 
	\]
	\[
	I_{n}^{2} =\left(2n+2\left( \begin{array}{c} n \\2 \end{array}\right)\right)\cdot\frac{\pi^{n}}{(n+3)!} ,
	\]
	\[
	I_{n}^{3} = \left(6n +12\left( \begin{array}{c} n \\2 \end{array}\right)+6\left( \begin{array}{c} n \\3  \end{array} \right) \right) \cdot \frac{\pi^{n}}{(n+3)!},
	\]
\end{lemma}
As explained in Section \ref{Sec:1}, the terms in the second-order obstruction can be seen as $SU_{n+1}$-invariant cubic polynomials and therefore a multiple of
\[
P_{0}(\eta):=\int_{\mathbb{CP}^{n}}f_{\eta}^{3} \ dV_{g},
\]  
where $\eta \in \mathfrak{su}_{n+1}.$ 
\begin{remark}
The integral is a (non-zero) multiple of the cubic polynomial 
\[
\sqrt{-1}\tr(\eta^{3}); 
\] 
in the case $\eta \in \mathfrak{su}_{2n+1}$, this polynomial has no non-trival zeros. (The fact that the multiple is non-zero was proved  in \cite{HMW}.) In the ${\mathbb{CP}^{2n+1}\times \mathbb{CP}^{1}}$	case, the zeros of the polynomial are easy to characterise (see \cite{BHMW}, \cite{NagSem23}, and \cite{HSS24}). Thus the set of EIDs that are unobstructed at second order is explicitly known and it might well be possible to prove that they are in fact obstructed at third order (see e.g. \cite{LZ23}) thus proving the rigidity of $\mathbb{CP}^{n}\times \mathbb{CP}^{1}$ in general.	
\end{remark}
The explicit form of the eigenfunction and the volume form in Equations (\ref{eqn:eigenfunc}) and (\ref{eqn:volform}) as well as the values of the integrals in Lemma \ref{lem:SIV} yields
\begin{equation}\label{eqn:f_3_int} 
\int_{\mathbb{CP}^{n}}f_{\gamma}^{3} \ dV_{g} = I_{n}^{3}-3nI_{n}^{2}+3n^{2}I_{n}^{1}-n^{3}I_{n}^{0} =2n(1-n^{2})\frac{\pi^{n}}{(n+3)!} .
\end{equation}
We can now compute the key quantities that Koiso needed.
\begin{lemma}[Identities (6.8.2), (6.8.10), and  (6.8.12) from Lemma 6.8 in \cite{KoiOsaka2}]
For any eigenfunction $f\in E_{1}$ we have
\begin{eqnarray}
\langle|\mathrm{Hess}f|^{2},f\rangle_{L^{2}} = 0,\label{eqn:Hess_f_ip}\\
\langle \nabla_{i}\nabla_{j} f \cdot \nabla ^{j}\nabla_{k}f,\nabla_{i}\nabla_{k}f \rangle_{L^{2}} = (n+1)^{3}\int_{\mathbb{CP}^{n}}f^{3} \ dV_{g}. \label{eqn:Hess_cube}
\end{eqnarray} 
For the eigenfunction $f_{\gamma} \in E_{1}$, we have
\begin{equation}\label{eqn:KoiHessid}
\Delta\mathrm{Hess}f_{\gamma} = 2(\mathrm{Hess}f_{\gamma}-(n+1)f_{\gamma}\cdot g).
\end{equation}
For any eigenfunction $f\in E_{1}$ we have
\begin{equation} \label{eqn:RoughLaplacianHess}
\langle \Delta(\mathrm{Hess} f+(n+1)fg), f\mathrm{Hess} f\rangle_{L^{2}} = -4n(n+1)^{2}\int_{\mathbb{CP}^{n}}f^{3} \ dV_{g}
\end{equation}

\end{lemma}
\begin{proof}
Using Equation (\ref{eqn:Hess_com_2}) we find, up to a constant that we need not worry about,
\[
|\mathrm{Hess}f|^{2} = \nabla_{k}\nabla_{\bar{l}}f\cdot\nabla^{k}\nabla^{\bar{l}}f.
\]
We now choose the eigenfunction to be $f_{\gamma}$ and calculate using (\ref{eqn:Hess_com_1}) and (\ref{eqn:FS_inv}) 
\[
|\mathrm{Hess}f_{\gamma}|^{2} = \frac{(n+1)^{2}}{(1+\|w\|^{2})^{2}}(n-2\|w\|^{2}+\|w\|^{4}).
\]
Thus, by (\ref{eqn:eigenfunc}) and the values of the integrals in Lemma \ref{lem:SIV}
 \[
\langle|\mathrm{Hess}f_{\gamma}|^{2},f_{\gamma}\rangle_{L^{2}} = (n+1)^{3}\left(I_{n}^{3}-(n+2)I_{n}^{2}+3nI_{n}^{1}-n^{2}I_{n}^{0} \right)=0.
\]
and an application of the standard argument on invariant polynomials proves the first identity. \\
\\
For the second, we again use (\ref{eqn:Hess_com_1}) and (\ref{eqn:FS_inv}) to
prove
\[
H^{k}_{l}:=\nabla^{k}\nabla_{l}f_{\gamma} = \frac{n+1}{1+\|w\|^{2}}\left( \delta_{kl}-\overline{w}_{k}w_{l}\right).
\]
Thus, the (complex coordinate version of the) integrand of the second identity is,
\[
\langle \nabla_{k}\nabla_{\bar{p}} f_{\gamma} \cdot \nabla ^{\bar{p}}\nabla_{\bar{l}}f_{\gamma},\nabla_{k}\nabla_{\bar{l}}f_{\gamma}\rangle = H^{k}_{l}H^{l}_{p}H^{p}_{k} = \frac{(n+1)^{3}}{(1+\|w\|^{2})^{3}}\left(n-3\|w\|^{2}+3\|w\|^{4}-\|w\|^{6}\right).
\]
Using the values in Lemma \ref{lem:SIV} yields
\[
\langle \nabla_{k}\nabla_{\bar{p}} f_{\gamma} \cdot \nabla ^{\bar{p}}\nabla_{\bar{l}}f_{\gamma},\nabla_{k}\nabla_{\bar{l}}f_{\gamma}\rangle_{L^{2}} =(n+1)^{3}\left( I_{n}^{0}-3I_{n}^{1}+3I_{n}^{2}-I_{n}^{3}\right) = (n+1)^{3}n(1-n^{2})\frac{\pi^{n}}{(n+3)!}.
\]
The identity follows from (\ref{eqn:f_3_int}) and multiplying by the appropriate factor from Lemma \ref{Lem:ComptoRiem}. \\
\\
The final identity follows by noting that
\[
\Delta \mathrm{Hess} f_{\gamma} = -2g^{p\bar{q}}\nabla_{{p}}\nabla_{\bar{q}}\nabla_{k}\nabla_{\bar{l}}f_{\gamma} = -2(n+1) f_{\gamma}g_{k\bar{l}}+2(\mathrm{Hess} f_{\gamma})_{k\bar{l}},
\]
where we have used the Equation (\ref{eqn:4Der}) in the final equality. Integratation, as well as an application of the first identity (\ref{eqn:Hess_f_ip}), and the standard argument about cubic polynomials yields the result. 
\end{proof}
We can now give a fairly easy proof of the following.
\begin{lemma}[Koiso, Lemma 6.9 in \cite{KoiOsaka2}]
For any eigenfunction $f\in E_{1}$ we have
\begin{eqnarray}
\langle\nabla_{i}\nabla_{j}\nabla_{k}\nabla_{l}f, \nabla_{i}\nabla_{j}f\cdot \nabla_{k}\nabla_{l}f \rangle_{L^{2}} = -2(n+1)^{3}\int_{\mathbb{CP}^{n}}f^{3} \ dV_{g}, \label{eqn:DHess_1} \\
\langle\nabla_{i}\nabla_{k}\nabla_{j}\nabla_{l}f, \nabla_{i}\nabla_{j}f\cdot \nabla_{k}\nabla_{l}f \rangle_{L^{2}} = -(n+1)^{3}\int_{\mathbb{CP}^{n}}f^{3} \ dV_{g}. \label{eqn:DHess_2}
\end{eqnarray}
\end{lemma}
\begin{proof}
For the first identity we use (\ref{eqn:4Der}) and (\ref{eqn:FS_inv}) to calculate 
\[
\langle\nabla_{{p}}\nabla_{\bar{q}}\nabla_{k}\nabla_{\bar{l}}f_{\gamma},\nabla_{{p}}\nabla_{\bar{q}}f_{\gamma}\cdot  \nabla_{{k}}\nabla_{\bar{l}}f_{\gamma}\rangle = (n+1)\langle|\mathrm{Hess}f_{\gamma}|^{2}, f_{\gamma}\rangle -\langle \nabla_{k}\nabla_{\bar{p}} f_{\gamma} \cdot \nabla ^{\bar{p}}\nabla_{\bar{l}}f_{\gamma},\nabla_{k}\nabla_{\bar{l}}f_{\gamma}\rangle.
\]
Using the standard argument about cubic polynomials, the result follows from (\ref{eqn:Hess_f_ip}) and (\ref{eqn:Hess_cube}) as well as noting the factor of 4 coming from Lemma \ref{Lem:ComptoRiem}.\\
\\
For the second identity, we note that as the Hessian is symmetric, we can also calculate,
\[
\langle\nabla_{i}\nabla_{k}\nabla_{l}\nabla_{j}f, \nabla_{i}\nabla_{j}f\cdot \nabla_{l}\nabla_{k}f \rangle_{L^{2}}
\] 
Thus in complex coordinates
\[
\langle\nabla_{p}\nabla_{\bar{k}}\nabla_{l}\nabla_{\bar{q}}f, \nabla_{p}\nabla_{\bar{q}}f\cdot \nabla_{l}\nabla_{\bar{k}}f \rangle = (n+1)\langle|\mathrm{Hess}f_{\gamma}|^{2}, f_{\gamma}\rangle -\langle \nabla_{k}\nabla_{\bar{p}} f_{\gamma} \cdot \nabla ^{\bar{p}}\nabla_{\bar{l}}f_{\gamma},\nabla_{k}\nabla_{\bar{l}}f_{\gamma}\rangle.
\]
Using the standard argument regarding cubic polynomials, the result follows from (\ref{eqn:Hess_f_ip}) and (\ref{eqn:Hess_cube}) as well as noting the factor of 2 coming from Lemma \ref{Lem:ComptoRiem}.
\end{proof}
We will need a final identity not explicitly used by Koiso in order to bypass some of the calculations of \cite{KoiOsaka2}
\begin{lemma}
	For any eigenfunction $f\in E_{1}$ we have
	\begin{equation}\label{eqn:notkoi_4d}
	\langle\nabla_{i}\nabla_{l}\nabla_{k}\nabla_{j}f,fg_{ij}\cdot\nabla_{k}\nabla_{l} f\rangle_{L^{2}} = -2(n+1)^{2}\int_{\mathbb{CP}^{n}}f^{3}dV_{g}
	\end{equation}
\end{lemma}
\begin{proof}
	
	We specialise to $f_{\gamma}$. Calculating in complex coordinates, we have, pointwise, by (\ref{eqn:4Der})
	\[
	\langle\nabla_{i}\nabla_{\bar{l}}\nabla_{k}\nabla_{\bar{j}}f_{\gamma},f_{\gamma}g_{i\bar{j}}\cdot\nabla_{k}\nabla_{\bar{l}} f_{\gamma}\rangle = -\langle g_{k\bar{j}}\nabla_{i}\nabla_{\bar{l}}f_{\gamma}+g_{k\bar{l}}\nabla_{i}\nabla_{\bar{j}}f_{\gamma},f_{\gamma}g_{i\bar{j}}\nabla_{k}\nabla_{\bar{l}}f_{\gamma} \rangle 
	\] 
	\[
	=-\langle |\mathrm{Hess} f_{\gamma}|^{2},f_{\gamma}\rangle - (n+1)^{2}f_{\gamma}^{3}.
	\]
	The result follows after integrating,  using (\ref{eqn:Hess_f_ip}), scaling by the factor 2 as in Lemma \ref{Lem:ComptoRiem}, and applying the standard argument about $SU_{n+1}$-invariant polynomials. 
\end{proof}	
\section{Computing the obstruction integral}
With the identities of the previous section in hand, we can now more-or-less follow Koiso's original calculation of the obstruction.  
\begin{proof}(of Koiso's Theorem)\\
As in \cite{KoiOsaka2}, we write
\[
\psi = \mathrm{Hess}f+(n+1)fg_{1} \qquad \mathrm{and} \qquad \varphi=(1-n)(n+1)fg_{0},
\]
where $f\in E_{1}$. This infinitesimal deformation associated to $f$ is then
\[
h = \psi +\varphi.
\]
The zeroth-order term is thus
\[
\langle h_{ik}h^{k}_{j},h_{ij} \rangle_{L^{2}} = \langle \psi_{ik}\psi^{k}_{j},\psi_{ij} \rangle_{L^{2}} + \langle \varphi_{ik}\varphi^{k}_{j},\varphi_{ij} \rangle_{L^{2}}.
\]
Calculating the term $\langle \psi_{ik}\psi^{k}_{j},\psi_{ij} \rangle_{L^{2}}$ yields (we omit the volume form for readability)
\[
\langle \nabla_{i}\nabla_{j}f, \nabla_{i}\nabla_{k}f\cdot \nabla^{k}\nabla_{j}f \rangle_{L^{2}}+3(n+1)\langle|\mathrm{Hess}f|^{2},f\rangle_{L^{2}}-6(n+1)^{3}\int_{\mathbb{CP}^{n}} f^{3}+2n(n+1)^{3}\int_{\mathbb{CP}^{n}} f^{3}.
\]
Using Equations (\ref{eqn:Hess_cube}) and (\ref{eqn:Hess_f_ip}) on the first two terms in the previous expression gives
\[
\langle \psi_{ik}\psi^{k}_{j},\psi_{ij} \rangle_{L^{2}} = (2n-5)(n+1)^{3}\int_{\mathbb{CP}^{n}} f^{3} \ dV_{g_{1}}.
\]
It is easy to see that
\[
\langle \varphi_{ik}\varphi^{k}_{j},\varphi_{ij} \rangle_{L^{2}} = 2(1-n)^{3}(n+1)^{3} \int_{\mathbb{CP}^{n}} f^{3} \ dV_{g_{1}},
\]
and thus
\begin{equation} \label{eqn:KoisoZeroOrder}
\langle h_{ik}h^{k}_{j},h_{ij} \rangle_{L^{2}} = -\left(2 n^3 - 6 n^2 + 4 n + 3\right)(n+1)^{3}\int_{\mathbb{CP}^{n}}f^{3} \ dV_{g_{1}}.
\end{equation}

The second term in Koiso's obstruction (\ref{eqn:Koiobs}) splits as

\[
\langle \nabla_{i}\nabla_{j}h_{kl},h_{ij}h_{kl}\rangle_{L^{2}} = \langle \nabla_{i}\nabla_{j}\psi_{kl},\psi_{ij}\psi_{kl}\rangle_{L^{2}} + \langle \nabla_{i}\nabla_{j}\phi_{kl},\psi_{ij}\phi_{kl}\rangle_{L^{2}}. 
\]
The first term in the splitting is given by 

\begin{multline*}
 \langle \nabla_{i}\nabla_{j}\psi_{kl},\psi_{ij}\psi_{kl}\rangle_{L^{2}} =\\
 \langle \nabla_{i}\nabla_{j}\psi_{kl},\nabla_{i}\nabla_{j}f\cdot\nabla_{k}\nabla_{l}f\rangle_{L^{2}} + (n+1) \langle \nabla_{i}\nabla_{j}\psi_{kl},\nabla_{i}\nabla_{j}f\cdot f(g_{1})_{kl}\rangle_{L^{2}}\\
+(n+1) \langle \nabla_{i}\nabla_{j}\psi_{kl},f(g_{1})_{ij} \cdot\nabla_{k}\nabla_{l}f\rangle_{L^{2}}+(n+1)^{2}\langle \nabla_{i}\nabla_{j}\psi_{kl},f^{2}(g_{1})_{ij} \cdot (g_{1})_{kl}\rangle_{L^{2}} \\
=\langle \nabla_{i}\nabla_{j}\nabla_{k}\nabla_{l}f,\nabla_{i}\nabla_{j}f\cdot \nabla_{k}\nabla_{l}f\rangle_{L^{2}}+(n+1)\langle\nabla_{i}\nabla_{j}f\cdot(g_{1})_{kl},\nabla_{i}\nabla_{j}f\cdot\nabla_{k}\nabla_{l}f\rangle_{L^{2}}
\end{multline*}
\[
-(n+1)\langle \Delta \psi_{kl}, f\nabla_{k}\nabla_{l}f\rangle_{L^{2}}-(n+1)^{2}\langle \Delta\mathrm{tr}(\psi),f^{2}\rangle_{L^{2}}.
\]
Using (\ref{eqn:DHess_1}), (\ref{eqn:Hess_f_ip}) yields
\begin{multline*}
\langle \nabla_{i}\nabla_{j}\psi_{kl},\psi_{ij}\psi_{kl}\rangle_{L^{2}} \\
=-2(n+1)^{3}\int_{\mathbb{CP}^{n}}f^{3}-(n+1)\langle \Delta\psi_{kl},f\nabla_{k}\nabla_{l}f\rangle_{L^{2}}-4(n-1)(n+1)^{4}\int_{\mathbb{CP}^{n}}f^{3}.
\end{multline*}
Using Equation (\ref{eqn:RoughLaplacianHess}) we have
\[
\langle \nabla_{i}\nabla_{j}\psi_{kl},\psi_{ij}\psi_{kl}\rangle_{L^{2}} = -2(n+1)^{3}(2n^{2}-n-1)\int_{\mathbb{CP}^{n}}f^{3} \ dV_{g_{1}}
\]
The second term can be computed as
\[
\langle \nabla_{i}\nabla_{j}\phi_{kl},\psi_{ij}\phi_{kl}\rangle_{L^{2}}  = 2(1-n)^{2}(n+1)^{2}\langle \nabla_{i}\nabla_{j}f,f\nabla_{i}\nabla_{j}f+(n+1)f^{2}(g_{1})_{ij}\rangle_{L^{2}},
\]
\[
=-4(1-n)^{2}(1+n)^{4}\int_{\mathbb{CP}^{n}} f^{3} \ dV_{g_{1}},
\]
where we have used Equation (\ref{eqn:Hess_f_ip}). Hence
\begin{equation}\label{eqn: KoisoFirstOrder1}
\langle \nabla_{i}\nabla_{j}h_{kl},h_{ij}h_{kl}\rangle_{L^{2}} =-2 (n + 1)^3 (2 n^3 - 4 n + 1)\int_{\mathbb{CP}^{n}} f^{3} \ dV_{g_{1}}
\end{equation}

The final term in (\ref{eqn:Koiobs}) splits as
\[
\langle\nabla_{i}\nabla_{l}h_{kj},h_{ij}h_{kl}\rangle_{L^{2}} = \langle\nabla_{i}\nabla_{l}\psi_{kj},\psi_{ij}\psi_{kl}\rangle_{L^{2}}  
\]
The fact that $\psi$ is divergence free yields
\begin{multline*}
\langle\nabla_{i}\nabla_{l}\psi_{kj},\psi_{ij}\psi_{kl}\rangle_{L^{2}} \\ = \langle\nabla_{i}\nabla_{l}\psi_{kj},\nabla_{i}\nabla_{j} f\cdot \nabla_{k}\nabla_{l} f \rangle_{L^{2}} + (n+1)\langle\nabla_{i}\nabla_{l}\psi_{kj},f(g_1)_{ij}\cdot\nabla_{k}\nabla_{l} f\rangle_{L^{2}}
\end{multline*}
The first term splits again, so
\begin{multline*}
\langle\nabla_{i}\nabla_{l}\psi_{kj},\nabla_{i}\nabla_{j} f\cdot \nabla_{k}\nabla_{l} f \rangle_{L^{2}} \\
= \langle\nabla_{i}\nabla_{l}\nabla_{k}\nabla_{j}f,\nabla_{i}\nabla_{j} f\cdot \nabla_{k}\nabla_{l} f \rangle_{L^{2}}+(n+1)\langle\nabla_{i}\nabla_{l}f\cdot (g_{1})_{kj},\nabla_{i}\nabla_{j} f\cdot \nabla_{k}\nabla_{l} f \rangle_{L^{2}}.
\end{multline*}
Using identities (\ref{eqn:DHess_2}) and (\ref{eqn:Hess_cube}) yields
\[
\langle\nabla_{i}\nabla_{l}\psi_{kj},\nabla_{i}\nabla_{j} f\cdot \nabla_{k}\nabla_{l} f \rangle_{L^{2}} = \left(-(n+1)^3+(n+1)^{4}\right)\int_{\mathbb{CP}^{n}}f^{3} \ dV_{g_{1}}.
\] 
The second term can be split
\begin{multline*}
\langle\nabla_{i}\nabla_{l}\psi_{kj},f(g_1)_{ij}\cdot\nabla_{k}\nabla_{l} f\rangle_{L^{2}}\\ = \langle\nabla_{i}\nabla_{l}\nabla_{k}\nabla_{j}f,f(g_1)_{ij}\cdot\nabla_{k}\nabla_{l} f \rangle_{L^{2}}+(n+1)\langle\nabla_{i}\nabla_{l}f\cdot(g_{1})_{kj}, f(g_1)_{ij}\cdot\nabla_{k}\nabla_{l} f \rangle_{L^{2}}.
\end{multline*}
Hence from Equation (\ref{eqn:notkoi_4d}) and (\ref{eqn:Hess_f_ip}) we obtain
\[
\langle\nabla_{i}\nabla_{l}\psi_{kj},f(g_1)_{ij}\cdot\nabla_{k}\nabla_{l} f\rangle_{L^{2}} = -2(n+1)^{2}\int_{\mathbb{CP}^{n}}f^{3} \ dV_{g_{1}}.
\]
Finally,
\begin{equation}\label{eqn:KoisoFirstOrder2}
\langle\nabla_{i}\nabla_{l}h_{kj},h_{ij}h_{kl}\rangle_{L^{2}} =((n+1)^{4}-3(n+1)^{3})\int_{\mathbb{CP}^{n}}f^{3} \ dV_{g_{1}}.
\end{equation}
We compute Koiso's obstruction (\ref{eqn:Koiobs}) using Equations (\ref{eqn:KoisoZeroOrder}), (\ref{eqn: KoisoFirstOrder1})), and (\ref{eqn:KoisoFirstOrder2}))
\begin{multline*}
\mathcal{I}(h) = \left(-2(n+1)^{4}(2n^{3}-6n^{2}+4n+3)- \right. \\
\left. 6(n+1)^{3}(2n^{3}-4n+1)-6(n+1)^{3}(n-2)\right) \int_{\mathbb{CP}^{n}}f^{3} \ dV_{g_{1}}.
\end{multline*}
Simplifying, yields
\[
\mathcal{I}(h) =-4n(n-1)(n+1)^{5}\int_{\mathbb{CP}^{n}}f^{3} \ dV_{g_{1}}.
\]
The theorem follows for ${\mathbb{CP}^{2n}\times \mathbb{CP}^{1}}$ as the integral does not vanish for any $f\in E_{1}$ and so all the EIDs are obstructed to second order.
\end{proof}

\begin{remark}
	We can compare our calculation of $\mathcal{I}(h)$ to the quantity on page 667 of \cite{KoiOsaka2} (written in the notation of \cite{KoiOsaka2}),
	\[
	\langle E''(h,h),h\rangle = -\frac{(n_{1}-2)(n_{1}+n_{2}-2)(n_{1}+2n_{2}-2)}{n_{2}^{2}}\mathcal{E}^{4}\langle f^{2},f\rangle.
	\]
	Unpacking this in our notation, $\langle E''(h,h),h\rangle = \frac{1}{2}\mathcal{I}(h)$, $n_{1}=2n$,  $n_{2}=2$, and $\mathcal{E}=(n+1)$. Thus we see that we recover exactly Koiso's original constant.
\end{remark}
\bibliographystyle{acm} 
\bibliography{RigidityRefs}

\begin{thebibliography}{10}

\bibitem{BHMW}
{\sc Batat, W., Hall, S.~J., Murphy, T., and Waldron, J.}
\newblock Rigidity of {$SU_{n}$}-{T}ype {S}ymmetric {S}paces.
\newblock {\em Int. Math. Res. Not. IMRN}, 3 (2024), 2066--2098.

\bibitem{Bes}
{\sc Besse, A.~L.}
\newblock {\em Einstein manifolds}.
\newblock Classics in Mathematics. Springer-Verlag, Berlin, 2008.
\newblock Reprint of the 1987 edition.

\bibitem{Hall24}
{\sc Hall, S.~J.}
\newblock The {F}ubini--{S}tudy metric on an `odd' {G}rassmannian is rigid.
\newblock {\em preprint (2403.18757)\/} (2024).

\bibitem{HMW}
{\sc Hall, S.~J., Murphy, T., and Waldron, J.}
\newblock Compact {H}ermitian symmetric spaces, coadjoint orbits, and the
  dynamical stability of the {R}icci flow.
\newblock {\em J. Geom. Anal. 31}, 6 (2021), 6195--6218.

\bibitem{HSS24}
{\sc Hall, S.~J., Schwahn, P., and Semmelmann, U.}
\newblock On the rigidity of the complex {G}rassmannians.
\newblock {\em preprint\/} (2024).

\bibitem{KoiOsaka1}
{\sc Koiso, N.}
\newblock Rigidity and stability of {E}instein metrics - the case of compact
  symmetric spaces.
\newblock {\em Osaka J. Math. 17\/} (1980), 51--73.

\bibitem{KoiOsaka2}
{\sc Koiso, N.}
\newblock Rigidity and infinitesimal deformability of {E}instein metrics.
\newblock {\em Osaka J. Math. 19\/} (1982), 643--668.

\bibitem{Koi_EMCS}
{\sc Koiso, N.}
\newblock Einstein metrics and complex structures.
\newblock {\em Invent. math 73\/} (1983), 71--106.

\bibitem{LZ23}
{\sc Li, Y., and Zhang, W.}
\newblock Rigidity of complex projective spaces in {R}icci shrinkers.
\newblock {\em Calc. Var. Partial Differential Equations 62}, 6 (2023), Paper
  No. 171, 30.

\bibitem{NagSem23}
{\sc Nagy, P.-A., and Semmelmann, U.}
\newblock Second order {E}instein deformations.
\newblock {\em preprint (2305.07391)\/} (2023).

\bibitem{SchSem24}
{\sc Schwahn, P., and Semmelmann, U.}
\newblock On the rigidity of the complex {G}rassmannians.
\newblock {\em preprint (2403.04681)\/} (2024).

\end{thebibliography}

\end{document}